\documentclass{amsart}

\usepackage{latexsym}
\usepackage{amsmath}
\usepackage{amssymb}

\usepackage{multicol}

\usepackage[english]{babel}

\theoremstyle{plain}

\newtheorem{theorem}{Theorem}[section]
\newtheorem{proposition}[theorem]{Proposition}
\newtheorem{lemma}[theorem]{Lemma}
\newtheorem{cor}[theorem]{Corollary}

\theoremstyle{definition}

\newtheorem{rem}[theorem]{Remark}

\newcommand{\Dbb}{\mathbb D}
\newcommand{\Tbb}{\mathbb T}

\newcommand{\spd}{\mathbb{S}^{d-1}}

\newcommand{\rd}{{\mathbb{R}}^d}

\newcommand{\bd}{\mathbb{B}^d}

\newcommand{\Aaa}{E}

\numberwithin{equation}{section}

\begin{document}

\date{}

\author{Evgueni Doubtsov}
\address{St.~Petersburg Department
of Steklov Mathematical Institute, Fo\-ntanka 27, St.~Petersburg 191023, Russia}
\email{dubtsov@pdmi.ras.ru}

\author{Ioann Vasilyev}
\address{St.~Petersburg Department
of Steklov Mathematical Institute, Fo\-ntanka 27, St.~Petersburg 191023, Russia}
\email{ivasilyev@pdmi.ras.ru}

\author{Grigory Voinov}
\address{Department of Mathematics and Computer Science,
St.~Petersburg State University,
Line 14th (Vasilyevsky Island), 29, St.~Petersburg 199178, Russia}
\email{grihnuhmih@gmail.com}

\title[Coulomb fields on the sphere]{
Sharp weighted $L^1$ bounds for Coulomb fields on the sphere}

\begin{abstract}
Let $d \ge 2$, let $\alpha_1,\dots, \alpha_n > 0$, and put $A = \sum_k \alpha_k$. We
determine, up to constants depending only on $d$, the smallest possible
$L^1$ norm in the unit ball $\mathbb{B}^d\subset \mathbb{R}^d$ of the Coulomb field generated by charges of
strengths $\alpha_k$ placed on the unit sphere $\mathbb{S}^{d-1}$. More precisely, we prove
\[
\inf_{x_1,\dots,x_n \in \mathbb{S}^{d-1}}
\int_{\mathbb{B}^d}
\left|
\sum_{k=1}^n
\alpha_k
\frac{x_k-x}{|x_k - x|^d}
\right|
dx \asymp_d A^{-1/(d-1)}
\sum_{k=1}^n
\alpha_k^{d/(d-1)}.
\]
The lower bound follows from an  $L^1$ Lipschitz trace estimate and an
explicit weighted tent function on the sphere. For the upper bound
we use a partition of the sphere into cells of masses $\alpha_k/A$ and
diameters $O_d \left((\alpha_k/A)^{1/(d-1)} \right)$.
For the unit weights the sharp order
is $n^{(d-2)/(d-1)}$, and in dimension three it is
$\sqrt{n}$.
\end{abstract}



\maketitle

\section{Introduction}
Let $\mathbb D$ denote the unit disc of the complex plane $\mathbb C$. 

\subsection{Chui's conjecture}
Consider points $z_1, z_2, \ldots, z_n$ on the unit circle $\Tbb=\partial \mathbb D$. The Chui conjecture, formulated in \cite{Chu71}, suggests that the quantity
	\begin{equation}\label{eq:chui}
		\int_{\Dbb}\Big| \sum_{k=1}^n \frac{1}{z-z_k} \Big|\, |dz|,
	\end{equation}
	where $|dz|$ denotes the Lebesgue measure on the complex plane, attains its minimum if the points $z_1, z_2, \ldots, z_n$ are uniformly distributed on the unit circle, i.e., if $z_k=e^{2\pi i k/n}$. This conjecture has the following natural physical interpretation: in order to minimize the average strength of the electrostatic field generated by $n$ unit charges on the circle, the charges should be distributed uniformly.
	
 Shortly after formulation of the conjecture, D.J.~Newman~\cite{New72} proved that this average strength of an electrostatic field is bounded from below, i.e., there exists an absolute constant $c > 0$ such that for every $n$ and any points $z_1, z_2, \ldots, z_n$ on the unit circle $\Tbb$ the following inequality holds:
	\begin{equation}\label{eq:newbound}
		\int_{\Dbb}\Big| \sum_{k=1}^n \frac{1}{z-z_k} \Big|\, |dz| \ge c.
	\end{equation}
	To be more precise, the above inequality holds with $c=\pi/18$. 

	
	A more recent progress was obtained in the paper \cite{ABF21} where the Chui conjecture was proved if one replaces the $L^1$ norm in the quantity~\eqref{eq:chui} with the norm in certain Bergman spaces.
See \cite{Co26} for two different proofs of such Chui conjecture and further results in this direction.
	
	Notice also that Chui's conjecture is closely related to the topic in Approximation Theory called the approximation by simplest fractions, see for instance papers \cite{Bor16}, \cite{BS23}, \cite{CZ23} and \cite{Kom23}.

\subsection{Weighted Chui's problem on the sphere}
In the present paper we study the following natural question: 
what happens if we place different charges (of the same sign) on the unit sphere of $\rd$, $d\ge 2$?
From the physical point of view, the case $d=3$ is the most interesting. 

Let $\bd_r$ be the open ball of radius $r$ in $\rd$ and let $\bd = \bd_1$ and $\spd = \partial\bd$. Given positive
weights $(\alpha_1,\dots, \alpha_n)=\alpha$ and points $(x_1,\dots, x_n)=X \in (\spd)^n$, define
\begin{equation}\label{e_1}
F_{\alpha,X}(x) =
\sum_{k=1}^n
\alpha_k
\frac{x_k-x}{
|x_k - x|^d}, \quad x \in \bd.
\end{equation}
Write
\begin{equation}\label{e_3}
A =
\sum_{k=1}^n
\alpha_k,  \quad
Q_d(\alpha) = A^{-1/(d-1)}
\sum_{k=1}^n
\alpha_k^{d/(d-1)}.
\end{equation}


The weighted lower bound
\begin{equation}\label{e_4}
\int_{\bd}
|F_{\alpha,X} (x)|\, dx \ge c_d
\frac{\sum_{k=1}^n \alpha_k^{1+2/d}}
{\sum_{k=1}^n \alpha_k^{2/d}}
\end{equation}
is proved in \cite{DTV}.
This result is a higher-dimensional weighted analogue of Newman’s
lower bound for the Chui problem \cite{New72}.
As shown in \cite{DTV}, estimate~\eqref{e_4} is sharp for $d=2$.

Our main result gives a better dependence for $d\ge 3$.
In fact, we have the following sharp result.

\begin{theorem}\label{c_12}
For every integer $d \ge 2$, there exists a dimensional constant $c_d > 0$
such that for every $n \ge 1$ and any $\alpha_1,\dots, \alpha_n > 0$,
\[
\int_{\mathbb{B}^d}
\left|
\sum_{k=1}^n
\alpha_k
\frac{x_k-x}{|x_k - x|^d}
\right|
dx \ge c_d Q_d(\alpha)
\]
for any points $x_1,\dots,x_n \in \mathbb{S}^{d-1}$.
Moreover, for every $d \ge 2$,
there exists a dimensional constant $C_d>0$, such 
that for any $\alpha_1,\dots, \alpha_n > 0$, it holds that  
\[
\int_{\mathbb{B}^d}
\left|
\sum_{k=1}^n \alpha_k
\frac{x_k-x}{|x_k - x|^d}
\right|
dx \leq C_d Q_d(\alpha)
\]
for some collection of points $x_1,\dots,x_n \in \mathbb{S}^{d-1}$.
\end{theorem}

In particular, we have the following result for the unit weights.

\begin{cor}\label{c_10}
For every $d \ge 2$, there exists a dimensional constant $c_d>0$
such that
\[
\int_{\mathbb{B}^d}
\left|
\sum_{k=1}^n
\frac{x_k-x}{|x_k - x|^d}
\right|
dx \ge c_d
n^{(d-2)/(d-1)}
\]
for any points $x_1,\dots,x_n \in \mathbb{S}^{d-1}$.
Moreover, for every $d \ge 2$, 
there exists a dimensional constant $C_d>0$ such that 
for each \(n \geq 1\) there exist
points $x_1,\dots,x_n \in \mathbb{S}^{d-1}$ with
\[
\int_{\mathbb{B}^d}
\left|
\sum_{k=1}^n
\frac{x_k-x}{|x_k - x|^d}
\right|
dx \leq C_d
n^{(d-2)/(d-1)}.
\]
In particular, the sharp order for $d = 3$ is
$\sqrt{n}$.
\end{cor}

The expression $Q_d(\alpha)$ improves the right-hand side of \eqref{e_4}.

\begin{proposition}\label{p_13}
For $d \ge 2$ and positive $\alpha_k$,
\begin{equation}\label{e_6}
        \frac{
        \sum_{k} \alpha_k^{1+2/d}}
        {\sum_k \alpha_k^{2/d}}
        \leq Q_d(\alpha).
\end{equation}
Equality holds identically when $d = 2$. When $d > 2$ and at least two weights
are present, the inequality is strict.
\end{proposition}

\begin{proof}
 Put $\ell_k = \alpha_k/A$, so that
$\sum_k \ell_k = 1$. Then the left-hand side of \eqref{e_6} is
\[
A
\frac{\sum_k
\ell_k^{1+2/d}}
{\sum_k
\ell_k^{2/d}
}, 
\]
and 
\[
Q_d(\alpha) = A
\sum_{k}
\ell_k^{d/(d-1)}.
\]
Since $2/d \le 1$, one has
$\sum_k \ell_k^{2/d}\geq 1$.
Moreover, $1 + 2/d \geq d/(d-1)$,
and therefore $\ell_k^{1+2/d}
\leq  \ell_k^{d/(d-1)}$.
This proves the inequality. The equality
statements follow from the same two comparisons.
\end{proof}

\subsection*{Organization of the paper}
In Section~\ref{s_low}, we obtain an $L^1$ Lipschitz trace inequality.
Using this estimate and an explicit Lipschitz tent function on the sphere,
we obtain the lower estimate in Theorem~\ref{c_12}.
Section~\ref{s_upp} is devoted to the proof of the upper estimate.
The corresponding argument uses a partition of the sphere into cells of masses $\alpha_k/A$ and
diameters $O_d \left((\alpha_k/A)^{1/(d-1)} \right)$.

\subsection*{Use of Artificial Intelligence}
Main results of Section~\ref{s_low} were obtained with assistance of ChatGPT~5.6. The authors have independently checked the resulting arguments and take full responsibility for the mathematical content.

\subsection*{Acknowledgements}
The authors would like to thank Anton Tselishchev for insightful discussions and valuable ideas.

\subsection*{Notation}
Everywhere below notation $X\lesssim_a Y$ means that $X\leq CY$ for some absolute constant $C>0$ that might depend only on a parameter $a$. We write $X\asymp_a Y$ once we have $X\lesssim_a Y$ and $Y\lesssim_a X$.

\section{Proof of the lower estimate}\label{s_low}


\subsection{Lipschitz trace inequality}
Let $|E|$ denote the surface measure of $E\subset \spd$.
Let $\sigma$ be normalized surface measure on $\spd$. Put
\[
\omega_{d-1} = |\spd|.
\]
We first assume $d \ge 3$. For $y\in\spd$, put
\begin{equation}\label{e_7}
        \Gamma_y(x) =
        \frac{1}{d-2}
        |x-y|^{2-d},\quad 
        K_y(x) = \nabla_x \Gamma_y (x) =
        \frac{y-x}{|y-x|^d}.
\end{equation}
Then
\begin{equation}\label{e_8}
-\Delta \Gamma_y = \omega_{d-1} \delta_y \quad \mathrm{in}\ \mathcal{D}^\prime(\rd).
\end{equation}
Also, observe that
 \begin{equation}\label{e_9}
\int_{\spd}
K_y(x)\, d\sigma(y) = 0, \quad x\in\bd.
\end{equation}
Indeed, the potential obtained by integrating $\Gamma_y$ against $\sigma$ is radial and
harmonic in the ball, hence constant there.

\medskip

The mechanism to prove lower estimate is the following trace inequality.

\begin{proposition}
\label{p_21} 
Let $\eta$ be a finite signed Borel
measure on $\spd$ with $\eta(\spd) = 0$. Define
\begin{equation}\label{e_10}
u(x) =
\int_{\spd}
\Gamma_y(x)\, d\eta(y).
\end{equation}
Then
\begin{equation}\label{e_11}
\sup_{\mathrm{Lip}(\varphi)\leq 1}
\left|
\int_{\spd}
\varphi\, d\eta
\right|
\leq C_d
\int_{\bd}
|\nabla u(x)|\, dx,
\end{equation}
where the Lipschitz constant $\mathrm{Lip}(\varphi)$ is computed using chordal distance.
\end{proposition}

\begin{proof}
For every compact $K \subset \rd$, Tonelli’s theorem gives
\begin{align*}
\int_K
|u(x)|\,dx  
&\le
\int_{\spd}
\int_K
|\Gamma_y(x)|\, dx d|\eta|(y), \\
\int_K
\left|
\int_{\spd}
K_y(x) d\eta(y)
\right|\,
dx 
&\le
\int_{\spd}
\int_K
|K_y(x)|\, dx d|\eta|(y).
\end{align*}
The inner integrals are bounded uniformly in $y$, because $|z|^{2-d}$ and $|z|^{1-d}$
are locally integrable. Testing against smooth functions and applying Fubini's theorem
therefore show that
\begin{equation}\label{e_12}
u \in W^{1,1}_{\mathrm{loc}} (\rd), 
\quad \nabla u(x) =
\int_{\spd}
K_y(x) d\eta(y), 
\quad -\Delta u = \omega_{d-1} \eta.
\end{equation}
Recall that $\eta(\spd) = 0$, thus $u(0) = 0$. Since $u$ is harmonic in $\bd$, its average on every
concentric ball $\bd_r$, $r < 1$, is zero. Letting $r \nearrow 1$ and using $u\in L^1(\bd)$ gives
\[
\int_{\bd} u(x)\, dx = 0.
\]
Therefore, Poincare’s inequality yields
\begin{equation}\label{e_13} 
\|u\|_{L^1(\bd)} \le C_d \|\nabla u\|_{L^1(\bd)}.
\end{equation}
For $x \neq 0$, put $x^\ast = x/|x|^2$. The identity $|x-y| = |x| |x^\ast -y|$ for $|y| = 1$
gives
\begin{equation}\label{e_14}
u(x) = |x|^{2-d} u(x^\ast),\quad |x| > 1.
\end{equation}
Differentiation provides
\begin{equation}\label{e_15}
        \nabla u(x) = 
        (2-d)|x|^{-d} u(x^\ast) x 
        + 
        |x|^{-d}\left(\nabla u(x^\ast) 
        -
        2\frac{\langle x, \nabla u(x^\ast)\rangle}{|x|^2}x \right).
\end{equation}
Inversion maps the set $1 < |x| < 2$ onto $1/2 < |x^\ast| < 1$. Changing variables in \eqref{e_15},
and then using \eqref{e_13}, gives
\begin{equation}\label{e_16}
        \int_{1<|x|<2}
        |\nabla u(x)|\, dx \le C_d
        \int_{\bd}
        |\nabla u(x)|\, dx.
\end{equation}

Now, let $\varphi$ be $1$-Lipschitz on the sphere. 
Adding a constant, normalize it so that $\|\varphi\|_\infty \le 2$.
By McShane's extension theorem \cite{McS34}, 
$\varphi$ can be extended to an 
$1$-Lipschitz function on $\rd$.
After this extension, we apply a fixed cutoff and obtain
a compactly supported Lipschitz function $\Phi$ with
\begin{equation}\label{e_17}
        \Phi_{|\spd} = \varphi,\quad
        \mathrm{supp\,} \Phi \Subset \bd_{3/2},
        \quad \|\nabla \Phi\|_\infty \le C_d.
\end{equation}

For $\varepsilon>0$ consider a function 
$\varrho\in C^\infty(\mathbb B^d)$ with $\int_{\mathbb B^d} \varrho=1$,
\(\varrho \ge 0\) and $\mathrm{supp}(\varrho)\subset \mathbb B^d$ 
and functions $\varrho_\varepsilon:=\varepsilon^{-d}\varrho(\cdot/\varepsilon)$. 
For standard mollifications $\Phi_\varepsilon$ defined by 
$\Phi_\varepsilon:=\Phi\ast \varrho_\varepsilon$, 
we have $\Phi_\varepsilon \to \Phi$ uniformly, 
$\nabla \Phi_\varepsilon \to \nabla\Phi$ almost everywhere,
and the gradients remain uniformly bounded. 

Green's formula, dominated convergence
against $|\nabla u| \in L^1(\bd_2)$ and \eqref{e_12} therefore give
\begin{align*}
\omega_{d-1}
\left|
\int_{\spd}
\varphi\, d\eta
\right|
=&                   
\left|
\int_{\rd}
\nabla u \cdot \nabla \Phi\,dx
\right|
\\
&\le C_d
\left(
\int_{\bd}
|\nabla u|\, dx +
\int_{1<|x|<2}
|\nabla u|\,dx
\right).
\end{align*}
Estimate \eqref{e_16} completes the proof.
\end{proof}

\begin{rem}
\label{r_22}
 For $d = 2$, take
$\Gamma_y(x) = -\log |x-y|$.
The averaged logarithmic potential
$\int_{\mathbb{S}^1} \Gamma_y(x)\, d\sigma(y)$ is radial and harmonic
in $\mathbb{B}^2$, hence constant there. 
Thus the identity \eqref{e_9} also holds in the
planar case. 
Then $\nabla_x \Gamma_y(x) = (y - x)/|y - x|^2$ and $-\Delta \Gamma_y= 2\pi \delta_y$. The
functions $| \log |z||$ and $|z|^{-1}$ are locally integrable in $\mathbb{R}^2$, so the preceding
$W^{1,1}$ argument applies. If $\eta(\mathbb{S}^1) = 0$, inversion gives
\[
-\log |x -y| = -\log |x| - \log |x^\ast - y|, \quad u(x) = u(x^\ast).
\]
The rest of the proof of Proposition~\ref{p_21} is unchanged, with $\omega_{d-1}$
replaced by $2\pi$.
\end{rem}

\subsection{A direct Lipschitz test for the lower bound}
We now prove the lower bound by inserting one explicit
function into Proposition~\ref{p_21}. 
\begin{proposition}\label{p_31}
For arbitrary $x_1, \dots, x_n \in\spd$ and $\alpha_1,\dots,\alpha_n >0$,
there is an $1$-Lipschitz function $\varphi: \spd \to [0,\infty)$ such that
\begin{equation}\label{e_18}
\int_{\spd}
\varphi\, d
\left(
\sum_{k=1}^n
\alpha_k \delta_{x_k} - A\sigma
\right)
\ge c_d Q_d(\alpha).
\end{equation}
\end{proposition}
\begin{proof}

A chordal cap satisfies
\begin{equation}\label{e_19}
\sigma\{y : |y-z| < r\} \le C_d r^{d-1},\quad z \in \spd, \quad 0 < r \le 1.
\end{equation}
Choose a number $0 < \varepsilon_d \le 1$ that may depend on the dimension $d$ only, to be fixed below, and set $\ell_k = \alpha_k/A$,
\begin{equation}\label{e_20}
r_k = \varepsilon_d (\alpha_k/A)^{1/(d-1)}, \quad
\theta_k(y) = (r_k- |y- x_k|)_{+}, \quad
\varphi(y) = \max_{1\le k \le n}
\theta_k(y).
\end{equation}
Every $\theta_k$ is $1$-Lipschitz for chordal distance, and the maximum of finitely
many $1$-Lipschitz functions is again $1$-Lipschitz. 
Thus, $\varphi$ is $1$-Lipschitz, moreover,
\begin{equation}\label{e_21}
\varphi(x_k) \ge r_k\quad (1 \le k \le n).
\end{equation}
Integration by level sets
and \eqref{e_19} give
\begin{equation}\label{e_22}
\begin{split}
\int_{\spd}
\theta_k(y)\, d\sigma(y) &=
\int^{r_k}_0
\sigma \{y : |y - x_k| < r_k - t\}\, dt \\
&\le C_d r_k^{d}.
\end{split}
\end{equation}
Since $\varphi \le
\sum_{k} \theta_k$, equations \eqref{e_21} and \eqref{e_22} imply
\begin{equation}\label{e_23}
\begin{split}
\int_{\spd}
\varphi\, d
\left(
\sum^n_{k=1}
\alpha_k \delta_{x_k} - A\sigma
\right)
&\ge
\sum^n_{k=1}
\alpha_k r_k - C_d A
\sum^n_{k=1}
r_k^{d}
 \\
&=
(\varepsilon_d - C_d\varepsilon^{d}_d)
A
\sum^n_{k=1}
\ell_k^{1+1/(d-1)}.
\end{split}
\end{equation}
Fix $\varepsilon_d$ so small that $C_d \varepsilon_d^{d-1} \le 1/2$. The right-hand side of \eqref{e_23} is then at
least
\[
\frac{\varepsilon_d}{2} A
\sum^n_{k=1}
\ell_k^{1+1/(d-1)}
=
\frac{\varepsilon_d}{2}
Q_d(\alpha),
\]
which proves \eqref{e_18}. 
\end{proof}

\begin{proof}[Proof of the lower estimate in Theorem~\ref{c_12}]
To finish the proof, put
\[
\eta =
\sum^n_{k=1}
\alpha_k \delta_{x_k} - A\sigma
\]
and let $u$ be its potential.
Identity \eqref{e_9} guarantees that 
$\nabla u = F_{\alpha, X}$ in $\bd$.
Apply Proposition \ref{p_21}, or Remark~\ref{r_22} when $d = 2$, to the function provided
by Proposition~\ref{p_31}. We obtain
\[
c_d Q_d(\alpha) \le
\left|
\int_{\spd}
\varphi \, d\eta
\right|
\le C_d
\int_{\bd}
|F_{\alpha,X}(x)|\, dx.
\]
This proves the claimed lower estimate for every configuration $X$.
\end{proof}

\section{Optimality of the lower estimate}\label{s_upp}
In this section, we prove the upper estimate in Theorem~\ref{c_12}.
We will use the following partition lemma.

\begin{lemma}[Voinov \cite{Voi26}] \label{l_31}
        Let \(n\) be a natural number. 
        Then for any \(\ell_1,\ldots,\ell_n > 0\) such that 
        \(\sum_1^n \ell_k=1\), there exist sets
        \(V_k\subset \spd\), \(1\leq k\leq n\), such that
        \[
                \bigcup_{k=1}^n V_k=\spd, \ \ \sigma(V_k)=\ell_k,
        \]
        the sets \(V_k\) are pairwise disjoint, and
        \[
                \mathrm{diam}(V_k)\asymp_d \ell_k^{\frac{1}{d-1}}.
        \]
\end{lemma}


We shall need the following 
estimate.

\begin{lemma}\label{l_32}
Let \(V\subset \mathbb S^{d-1}\subset \mathbb R^d\) be a 
measurable set with
\(0<\sigma(V)<\infty\). Put
\[
        d\mu(y)=\frac{1}{\sigma(V)}\,d\sigma(y)
        \qquad \text{ and } \qquad
        m=\int_V y\,d\mu(y).
\]
Let \(x\in\mathbb R^d\) be such that \(\delta:=\mathrm{dist}(x,\mathrm{conv} V)>0.\)
Then
\[
        \left|
        \int_V \frac{y - x}{|y - x|^d}\,d\mu(y)
        -
        \frac{m - x}{|m - x|^d}
        \right|
        \leq
        C_d\,\frac{1-|m|^2}{\delta^{d+1}}.
\]
\end{lemma}

Here and elsewhere in the text, \(\mathrm{conv}\) stands 
for the convex hull in \(\mathbb R^d\).

\begin{proof}
        We consider the vector-valued function
        \[
                K_z(x)=\frac{z-x}{|z-x|^d},
        \]
        where \(z\in\mathbb R^d\setminus\{x\}\). Since \(m\in \mathrm{conv} V\), the segment
        \[
                \{m+t(y-m):0\leq t\leq 1\}
        \]
        is contained in \(\mathrm{conv} V\) for every \(y\in V\). Hence this
        segment stays at distance at least \(\delta\) from \(x\).

        By Taylor's formula in \(\mathbb R^d\), applied componentwise to \(K_z(x)\),
        we have
        \[
                K_y(x)
                =
                K_m(x)
                +
                D K_m(x)[y-m]
                + 
                \int_0^1 (1-t)
                D^2 K_{m+t(y-m)}(x)[y-m,y-m]\,dt .
        \]
        Integrating over \(V\) with respect to \(d\mu\), the linear term vanishes,
        because
        \[
                \int_V (y-m)\,d\mu(y)
                =
                \int_V y\,d\mu(y)-m
                =
                0.
        \]
        Therefore,
        \[
                \int_V K_y(x)\,d\mu(y)-K_m(x)
                = 
                \int_V\int_0^1 (1-t)
                D^2K_{m+t(y-m)}(x)[y-m,y-m]\,dt\,d\mu(y).
        \]

        Consequently,
        \begin{equation}\label{e_tay}
        \left|
        \int_V K_y(x)\,d\mu(y)-K_m(x)
        \right|
        \leq
        \frac{1}{2}
        \sup_{z\in \mathrm{conv} V}\|D^2K_z(x)\|
        \int_V |y-m|^2\,d\mu(y).
        \end{equation}
        For the function \(K_z(x)\), a direct differentiation gives
        \[
                \|D^2K_z(x)\|
                \leq
                \frac{C_d}{|z-x|^{d+1}}.
        \]
        Since \(\mathrm{dist}(x,\mathrm{conv} V)=\delta\), it follows that
        \begin{equation}\label{e_d2}
                \sup_{z\in \mathrm{conv} V}\|D^2K_z(x)\|
                \leq
                \frac{C_d}{\delta^{d+1}}.
        \end{equation}
        Finally, since \(|y|=1\) for \(y\in \mathbb S^{d-1}\), we know that
        \(|y-m|^2=1+|m|^2-2\langle y,m\rangle\). 
        We recall the definition of \(m\) and integrate this equality to 
        infer the following properties
        \begin{equation}\label{e_bary}
                \int_V |y-m|^2\,d\mu(y)
                =
                \int_V |y|^2\,d\mu(y)-|m|^2
                =
                1-|m|^2.
        \end{equation}
        Combining the estimates \eqref{e_tay}, \eqref{e_d2} and \eqref{e_bary}, 
        we obtain the required inequality.
\end{proof}

\medskip

Let \(\overline{\mathbb B^d}\) denote the closed unit ball.
We shall also need the following elementary estimate. 

\begin{lemma} \label{l_ab}
        If \(a,b\in \overline{\mathbb B^d} \) and
        \(h:=|a-b|\), then
        \begin{equation}\label{e_ab}
                \int_{\mathbb B^d}
                \left|
                K_a(x)
                -
                K_b(x)
                \right|\,dx
                \leq C_d \sqrt{h}.
        \end{equation}
\end{lemma}

\begin{proof}
        Suppose first that \(h \geq \frac{1}{2}\).
        In this case, the direct estimate gives 
        \[ 
        \begin{split}
                &\int_{\mathbb B^d}
                \left|
                K_a(x)
                -
                K_b(x)
                \right|\,dx \\
                &\qquad\leq
                \int_{\mathbb B^d} |K_a(x)|\,dx +
                \int_{\mathbb B^d} |K_b(x)|\,dx 
                \le C_d,
        \end{split}
        \]
        which gives required inequality, since 
        \(\sqrt{h} \geq 1/\sqrt{2}\).

        \smallskip

        Now, suppose that \(h \leq \frac{1}{2}\).
        On the set \(B(a,2h)\cup B(b,2h)\), we estimate both fractions directly:
        \[
        \begin{split}
                &\int_{B(a,2h)\cup B(b,2h)}
                \left|
                K_a(x)
                -
                K_b(x)
                \right|\,dx \\
                &\qquad\leq
                \int_{B(a,2h)\cup B(b,2h)}
                \bigl(|K_a(x)|+|K_b(x)|\bigr)\,dx \\
                &\qquad\leq 2 C_d \int_0 ^ {2h} dt
                \leq C_d h.
        \end{split}
        \]
        Next, notice that outside the above set, the segment \(T\) joining 
        \(a\) to \(b\) stays at distance comparable to \(|x-a|\) from \(x\). 
        In other words, for \(x\in\mathbb B^d \setminus\bigl(B(a,2h)\cup B(b,2h)\bigr)\) 
        and \(\gamma_s\in T\) with \(\gamma_s=(1-s)a+sb\) and \(s\in[0,1]\), 
        we have \(|\gamma_s-a|\leq h\) and
        \(|x-a|\geq 2h\), which means that \(|x-\gamma_s|\geq \frac12|x-a|\). 
        Hence, by the Lagrange mean value inequality applied to the function 
        \(y \mapsto K_y(x)\), we have
        \[
                |K_a(x)-K_b(x)|
                \leq |a-b|\sup_{0\leq s\leq1}\|DK_{\gamma_s}(x)\|
                \leq C_d\frac{h}{|x-a|^d}.
        \]
        Thus
        \[
                \int_{\mathbb B^d\setminus(B(a,2h)\cup B(b,2h))}
                |K_a(x)-K_b(x)|\,dx
                \leq C_d h\int_{2h}^{2}\frac{dt}{t}
                \leq C_d h\log\frac1h
                \leq C_d \sqrt{h}
        \]
        and the lemma follows.
\end{proof}

\begin{proof}[Proof of the upper estimate in Theorem~\ref{c_12}]
        
        Put \(\ell_k := \alpha_k/A\). Since \(\ell_k > 0\) and 
        \(\sum_{k = 1} ^ n \ell_k = 1\),
        we may apply Lemma~\ref{l_31} and consider 
        a partition into the sets \(V_k\). 
        Let \(m_k\) denote the Euclidean barycenter of \(V_k\), that is,
        \[
                m_k
                :=
                \frac{1}{\sigma(V_k)}
                \int_{V_k} y\,d\sigma(y)
                =
                \frac{1}{\ell_k}\int_{V_k} y\,d\sigma(y).
        \]
        
        \smallskip

        First, we estimate the quantity
        \[
                J
                :=
                \int_{\mathbb B^d}
                \left|
                \sum_{k=1}^n
                \alpha_k K_{m_k}(x)
                \right|\,dx.
        \]
        Since
        \[
                \int_{\mathbb S^{d-1}}
                K_y(x)\,d\sigma(y)=0
        \]
        for
        \(x\in\mathbb B^d\),
        we may write
        \[
        \begin{split}
                J
                &=
                \int_{\mathbb B^d}
                \left|
                \sum_{k=1}^n
                \left[
                \alpha_k K_{m_k}(x)
                -
                A\int_{V_k}K_y(x)\,d\sigma(y)
                \right]
                \right|\,dx                                      \\
                &\leq
                A\sum_{k=1}^n
                \ell_k\int_{\mathbb B^d}
                \left|
                K_{m_k}(x)
                -
                \frac{1}{\ell_k}\int_{V_k}K_y(x)\,d\sigma(y)
                \right|\,dx,
        \end{split}
        \]
        Thus, it remains to estimate each summand.

        Let $B(x, r)$ denote the open ball of radius $r$ and centered at $x\in \rd$.
        Fix $k$. Let $\Aaa>1$ be a sufficiently large constant. We split
        the corresponding integrals into a near and a far part:
        \[
        \begin{split}
                I_k
                &:=
                \int_{\mathbb B^d \cap B(m_k,\Aaa\ell_k^{1/(d-1)})}
                \left|
                K_{m_k}(x)
                -
                \frac{1}{\ell_k}\int_{V_k}K_y(x)\,d\sigma(y)
                \right|\,dx,                                    \\
                II_k
                &:=
                \int_{\mathbb B^d\setminus B(m_k,\Aaa \ell_k^{1/(d-1)})}
                \left|
                K_{m_k}(x)
                -
                \frac{1}{\ell_k}\int_{V_k}K_y(x)\,d\sigma(y)
                \right|\,dx.
        \end{split}
        \]

        We estimate the term \(I_k\). By the triangle inequality,
        \[
                I_k
                \leq
                \int_{\mathbb B^d\cap B(m_k,\Aaa\ell_k^{1/(d-1)})}
                \frac{dx}{|m_k-x|^{d-1}}
                +
                \frac{1}{\ell_k}\int_{\mathbb B^d\cap B(m_k,\Aaa\ell_k^{1/(d-1)})}
                \int_{V_k}
                \frac{d\sigma(y)\,dx}{|y-x|^{d-1}},
        \]
        and thus the first term is bounded from above by
        \[
                C_d\int_0^{\Aaa\ell_k^{1/(d-1)}}\frac{t^{d-1}}{t^{d-1}}\,dt
                \leq C_{d,\Aaa}\ell_k^{1/(d-1)} .
        \]
        Moreover, since \(\mathrm{diam}(V_k) \le C_d\ell_k^{1/(d-1)}\) and
        \(m_k\in \mathrm{conv} V_k\), for every \(y\in V_k\) we have
        \[
                B(m_k,\Aaa\ell_k^{1/(d-1)})\subset B(y,C_\Aaa\ell_k^{1/(d-1)}).
        \]
        Therefore
        \[
        \begin{split}
                \frac{1}{\ell_k}\int_{\mathbb B^d \cap B(m_k,\Aaa\ell_k^{1/(d-1)})}
                \int_{V_k}
                \frac{d\sigma(y)\,dx}{|y-x|^{d-1}}
                &\leq
                \frac{1}{\ell_k}\int_{V_k}
                \int_{B(y,C_\Aaa\ell_k^{1/(d-1)})}
                \frac{dx}{|y-x|^{d-1}}\,d\sigma(y)                 \\
                &\leq
                C_{d,\Aaa} \frac{1}{\ell_k}\int_{V_k}\ell_k^{1/(d-1)}\,d\sigma(y)                   \\
                &=
                C_{d,\Aaa} \frac{1}{\ell_k}\sigma(V_k)\ell_k^{1/(d-1)}                              \\
                &=
                C_{d,\Aaa}\ell_k^{1/(d-1)}.
        \end{split}
        \]
        Hence
        \begin{equation}\label{e_Ik}
                I_k\leq C_{d,\Aaa}\ell_k^{1/(d-1)}.
        \end{equation}

        We now estimate \(II_k\). For fixed \(x\), apply the barycenter
        estimate from Lemma~\ref{l_32} to the set \(V=V_k\) and to the point \(m=m_k\). 
        This gives
        \[
                \left|
                K_{m_k}(x)
                -
                \frac{1}{\ell_k}\int_{V_k}K_y(x)\,d\sigma(y)
                \right|
                \leq
                C_d\,
                \frac{1-|m_k|^2}
                {\mathrm{dist}(x,\mathrm{conv} V_k)^{d+1}}.
        \]
        Since \(\mathrm{diam}(V_k) \leq  C_d\ell_k^{1/(d-1)}\), we have
        \[
                1-|m_k|^2
                =
                \frac{1}{\ell_k}\int_{V_k}|y-m_k|^2\,d\sigma(y)
                \leq
                C\ell_k^{2/(d-1)} .
        \]
        Furthermore, choosing $\Aaa$ sufficiently large, we have for
        $x\notin B(m_k,\Aaa\ell_k^{1/(d-1)})$,
        \[
                \mathrm{dist}(x, \mathrm{conv} V_k)
                \geq
                c\,|x-m_k|.
        \]
        Indeed, since \(m_k\in \mathrm{conv} V_k\) and
        \(\mathrm{diam}(\mathrm{conv} V_k) \leq C_d\ell_k^{1/(d-1)}\), every
        \(\xi\in \mathrm{conv} V_k\) satisfies
        \[
                |\xi-m_k|\leq C\ell_k^{1/(d-1)}.
        \]
        Thus, for $x\notin B(m_k,\Aaa\ell_k^{1/(d-1)})$,
        \[
                |x-\xi|
                \geq
                |x-m_k|-|\xi-m_k|
                \geq
                \left(1-\frac{C}{\Aaa}\right)|x-m_k|
                \geq
                c|x-m_k|,
        \]
        provided $\Aaa$ is large enough.

        Consequently, for \(x\in\mathbb B^d\setminus B(m_k,\Aaa\ell_k^{1/(d-1)})\),
        \[
                \left|
                K_{m_k}(x)
                -
                \frac{1}{\ell_k}\int_{V_k}K_y(x)\,d\sigma(y)
                \right|
                \leq
                C_d\frac{\ell_k^{2/(d-1)}}{|x-m_k|^{d+1}}.
        \]
        Using polar coordinates, we have
        \begin{equation}\label{e_IIk}
                \begin{split}
                II_k
                &\leq
                C_d\ell_k^{2/(d-1)}
                \int_{\mathbb B^d \setminus B(m_k,\Aaa\ell_k^{1/(d-1)})}
                \frac{dx}{|x-m_k|^{d+1}}                            \\
                &\leq
                C_d\ell_k^{2/(d-1)}
                \int_{\Aaa\ell_k^{1/(d-1)}}^{2}
                \frac{t^{d-1}}{t^{d+1}}\,dt                         \\
                &=
                C_d\ell_k^{2/(d-1)}
                \int_{\Aaa\ell_k^{1/(d-1)}}^{2}\frac{dt}{t^2}       \\
                &\leq
                C_{d,\Aaa}\ell_k^{1/(d-1)}.
                \end{split}
        \end{equation}

        Combining the estimates \eqref{e_Ik} and \eqref{e_IIk} 
        for \(I_k\) and \(II_k\), respectively, we obtain
        \[
                I_k+II_k\leq C_d\ell_k^{1/(d-1)}.
        \]
        Summing over \(k=1,\dots,n\), we get
        \begin{equation}\label{e_J}
                J
                \leq
                C_d A\sum_{k=1}^n \ell_k^{1+1/(d-1)}
                =
                C_d A^{-1/(d-1)} \sum_{k=1}^n\alpha_k^{d/(d-1)}.
        \end{equation}

        \bigskip
        It remains to replace the points \(m_k\), 
        which lie inside the ball, by points on the sphere.
        For \(m_k \neq 0\), define
        \[
                x_k:= \frac{m_k}{|m_k|}.
        \]
        If \(m_k = 0\), let \(x_k\) be an arbitrary
        point on the unit sphere.

        We claim that the replacement of \(m_k\) by \(x_k\) 
        does not change the order of the main estimate.

        First, recall that we already know the estimate 
        \(1-|m_k|^2\leq C\ell_k^{2/(d-1)}\). Therefore,
        \[
                |x_k-m_k|
                =1-|m_k|
                =\frac{1-|m_k|^2}{1+|m_k|}
                \leq C\ell_k^{2/(d-1)}.
        \]

        Applying Lemma~\ref{l_ab} with \(a=x_k\), \(b=m_k\), and using
        \(|x_k-m_k|\leq C\ell_k^{2/(d-1)}\), we get
        \[
        \begin{split}
                \int_{\mathbb B^d}
                \left|
                K_{x_k}(x)
                -
                K_{m_k}(x)
                \right|\,dx
                \leq
                C_d \ell_k^{1/(d-1)}.
        \end{split}
        \]
        Summing over \(k = 1, \dots, n\), we obtain 
        \[ 
                \int_{\mathbb B^d}
                \left|
                \sum_{k=1}^n
                \alpha_k
                \big(
                K_{x_k}(x)
                -
                K_{m_k}(x)
                \big)
                \right|\,dx
                \leq C_d \sum_{k=1}^n \alpha_k \ell_k^{1/(d-1)}.
        \]
        Consequently, 
        \[
                \int_{\mathbb B^d}
                \left|
                \sum_{k=1}^n \alpha_k K_{x_k}(x)
                \right|\,dx
                \leq J + C_d \sum_{k=1}^n\alpha_k \ell_k^{1/(d-1)}
                = J + C_d A^{-1/(d-1)} \sum_{k=1}^n\alpha_k^{d/(d-1)}.   
        \]
        By the already obtained estimate \eqref{e_J} on \(J\), this yields                
        \[
                \int_{\mathbb B^d}
                \left|
                \sum_{k=1}^n
                \alpha_k
                K_{x_k}(x)
                \right|\,dx
                \leq C_d A^{-1/(d-1)} \sum_{k=1}^n \alpha_k ^{d/(d-1)}.
        \]

        This proves the upper estimate in Theorem~\ref{c_12}.
\end{proof}

\bibliographystyle{amsplain}

\end{document}